\documentclass[12pt]{article}

\usepackage{amsmath}
\usepackage{amsthm}
\usepackage{amsfonts}
\usepackage{amssymb}
\usepackage{pb-diagram}

\def\afooter#1#2{{\def\thefootnote{\mbox{${}^{#1}$}}\mbox{\footnotemark[0]}
        \footnotetext{#2} }}

\newcommand\mc[1]{\mathcal{#1}}
\newcommand\Z{\mathbb{Z}}
\newcommand\Q{\mathbb{Q}}
\newcommand{\dlra}[1]{\stackrel{#1}{\longrightarrow}}
\newcommand\tensor{\otimes}
\newcommand\del{\partial}
\newcommand\hra{\hookrightarrow}
\newcommand\ra{\rightarrow}
\newcommand\sra{\ra \! \! \! \! \! \ra}

\newtheorem{Theorem}{Theorem}[section]
\newtheorem{Corollary}[Theorem]{Corollary}
\newtheorem{Remark}[Theorem]{Remark}
\newtheorem{Lemma}[Theorem]{Lemma}
\newtheorem{Proposition}[Theorem]{Proposition}

\advance \topmargin by -\headheight
\advance \topmargin by -\headsep
     
\evensidemargin \oddsidemargin
\title{The smooth structure set of $S^p \times S^q$
\afooter{\mbox{ \ }}{Keywords and phrases: smooth structure set, surgery exact sequence, product of spheres, diffeomorphism classification.\\ \hspace*{5mm} 2000 Mathematics
Subject Classification:  57R55,57R65.}} 
\author{Diarmuid Crowley}
\date{\today}

\begin{document}
\maketitle

\begin{abstract}
We calculate $\mathcal{S}^{Diff}(S^p \times S^q)$, the smooth structure set of $S^p \times S^q$, for $p, q \geq 2$ and $p+q \geq 5$.  As a consequence we show that in general $\mathcal{S}^{Diff}(S^{4j-1} \times S^{4k})$ cannot admit a group structure such that the smooth surgery exact sequence is a long exact sequence of groups.  We also show that the image of forgetful map $\mathcal{S}^{Diff}(S^{4j} \times S^{4k}) \rightarrow \mathcal{S}^{Top}(S^{4j} \times S^{4k})$ is not in general a subgroup of the topological structure set.
\end{abstract}

\section{Introduction}
We work in the categories of closed, oriented, simply-connected $Cat$-manifolds $M$ and $N$ of dimension $n \geq 5$ and orientation preserving maps, where $Cat = Diff$ for smooth manifolds or $Cat = Top$ for topological manifolds.  The $Cat$-structure set of $M$, $\mathcal{S}^{Cat}(M)$, is the set of structure invariants $[N, f]$ which are equivalence classes of homotopy equivalences $f: N \rightarrow M$ where $f_0 : N_0 \rightarrow M$ and $f_1: N_1 \rightarrow M$ are equivalent if $f_1^{-1} \circ f_0$ is homotopic to an isomorphism (diffeomorphism or homeomorphism).  The base point of $\mathcal{S}^{Cat}(M)$ is the equivalence class of ${\rm Id}_{}: M = M$.   There is an an obvious forgetful map $F: \mathcal{S}^{Diff}(M) \rightarrow \mathcal{S}^{Top}(M)$ such that the smooth and topological surgery exact sequences for $M$ are commuting long exact sequences of pointed sets 
\begin{equation} \label{seseqn1}
\begin{diagram}
\divide\dgARROWLENGTH by 4
\node{\dots} \arrow{e} \node{\mc{N}^{Diff}(M \times [0, 1])} \arrow{e,t}{} \arrow{s,r}{F} \node{L_{n+1}(e)} \arrow{e,t}{\omega^{Diff}} \arrow{s,r}{=} \node{\mathcal{S}^{Diff}(M)} \arrow{e,t}{\eta^{Diff}} \arrow{s,r}{F} \node{\mathcal{N}^{Diff}(M)} \arrow{s,r}{F} \arrow{e,t}{\theta^{Diff}} \node{L_{n}(e)} \arrow{s,r}{=}\\
\node{\dots} \arrow{e} \node{\mc{N}^{Top}(M \times [0,1])} \arrow{e,t}{} \node{L_{n+1}(e)} \arrow{e,t}{0} \node{\mathcal{S}^{Top}(M)} \arrow{e,t}{\eta^{Top}} \node{\mathcal{N}^{Top}(M)} \arrow{e,t}{\theta^{Top}} \node{L_{n}(e).}
\end{diagram}
\end{equation}
Here $\mathcal{N}^{Cat}(M)$ and $\mc{N}^{Cat}(M \times [0, 1])$ are the $Cat$ normal invariant sets of $M$ and $M \times [0, 1]$ relative boundary, $L_{n+1}(e)$ and $L_n(e)$ are the surgery obstruction groups: $L_i(e) \cong \Z, 0,\Z_2, 0$ as $i \equiv 0, 1, 2, 3$ mod $4$ and $L_{n+1}(e)$ acts transitively on the fibres of $\eta^{Cat}$.  Using identity maps as base points we have Sullivan's familiar identifications $\mathcal{N}^{Cat}(M) \equiv [M, G/Cat]$ and $\mc{N}^{Cat}(M \times [0, 1]) \equiv [\Sigma M, G/Cat]$ where $\Sigma M$ is the suspension of $M$.  We refer 
the reader to Section \ref{surgprelimsec} for some further definitions and references to the literature.

In this paper we calculate $\mc{S}^{Diff}(S^p \times S^q)$, for $p, q \geq 2$ and $n = p +q \geq 5$.  We develop the necessary preliminaries by first recalling $(\ref{seseqn1})$ when $M = S^n$ and $(\ref{seseqn1})^{Top}$ when $M = S^p \times S^q$.

The Generalised Poincar\'{e} Conjecture, due to \cite{Sm}, asserts that $\mc{S}^{Top}(S^n) = \{ [\rm Id] \}$ whereas the smooth structure set of $S^n$, $\mathcal{S}^{Diff}(S^n) = \Theta_n \cong \pi_n(Top/O)$, is the finite abelian group of diffeomorphism classes of homotopy $n$-spheres.  With $\mc{N}^{Diff}(S^n) \equiv \pi_n(G/O)$ and $\mc{N}^{Top}(S^n) \equiv \pi_n(G/Top)$ we have the commuting long sequences of abelian groups, the upper due to  \cite{K-M},
\begin{equation} \label{seseqnS^n}
\begin{diagram}
\divide\dgARROWLENGTH by 2
\node{\dots} \arrow{e} \node{\pi_{n+1}(G/O)} \arrow{e,t}{\theta^{Diff}} \arrow{s,r}{F} \node{L_{n+1}(e)} \arrow{e,t}{\omega^{Diff}} \arrow{s,r}{=} \node{\Theta_n} \arrow{e,t}{\eta^{Diff}} \arrow{s,r}{F} \node{\pi_n(G/O)} \arrow{s,r}{F} \arrow{e,t}{\theta^{Diff}} \node{L_{n}(e)} \arrow{s,r}{=}\\
\node{\dots} \arrow{e} \node{\pi_{n+1}(G/Top)} \arrow{e,t}{\theta^{Top}} \node{L_{n+1}(e)} \arrow{e,t}{0} \node{\{ [\rm Id] \} } \arrow{e,t}{0} \node{\pi_n(G/Top)} \arrow{e,t}{\theta^{Top}} \node{L_{n}(e).}
\end{diagram}
\end{equation}
The topological sequence gives the fundamental identification $\pi_n(G/Top) = L_n(e)$ which we often make without further comment.  The image of $\omega^{Diff}$ is $bP_{n+1}$, the finite cyclic group of diffeomorphism classes of homotopy spheres bounding parallelisable manifolds: thus $bP_{n+1} \cong L_{n+1}(e)/\theta^{Diff}(\pi_{n+1}(G/O)) \cong \pi_{n+1}(G/Top)/F(\pi_{n+1}(G/O))$.  

Now let $i: S^p \vee S^q \rightarrow S^p \times S^q$ be the inclusion and $c : S^p \times S^q \rightarrow S^{p+q}$ the collapse map.  We have the identification $\mathcal{N}^{Cat}(S^p \times S^q) \equiv [S^p \times S^q, G/Cat]$ and the split exact sequence
\[ 0 \dlra{} \pi_{p+q}(G/Cat) \dlra{c^*} [S^p \times S^q, G/Cat] \dlra{i^*} \pi_p(G/Cat) \times \pi_q(G/Cat) \dlra{} 0.\]
In Section \ref{surgprelimsec} we use the product of normal maps to define a section $\pi_{p, q}$ to $i^*$ and thus an identification
\begin{equation} \label{nmlspliteqn}
\pi_{p, q} \times c^* : \prod_{i=1}^3\pi_{p_i}(G/Cat)  \cong  [S^p \times S^q, G/Cat],~~(p_1, p_2, p_3) = (p, q, p+q).
\end{equation}
%
%
In the topological case, see \cite{Ra1, K-L}[Ex. 20.4, \S 7], the map $i^* \circ \eta^{Top}$ defines a bijection
\begin{equation} \label{topSS^pxS^qeqn}
 i^* \circ \eta^{Top} : \mc{S}^{Top}(S^p \times S^q) \equiv \pi_p(G/Top) \times \pi_q(G/Top) 
 \end{equation}
and the surgery exact sequences of $(\ref{seseqn1})$, the lower sequence found in\cite{Ra4}, are isomorphic to
\begin{equation} \label{seseqnS^pxS^q}
\begin{diagram}
\divide\dgARROWLENGTH by 2
\node{\prod_{i=1}^3\pi_{p_i+1}(G/O)} \arrow{e} \arrow{s,r}{F} \node{L_{p+q+1}(e)} \arrow{e,t}{\omega^{Diff}} \arrow{s,r}{=} \node{\mc{S}^{Diff}(S^p \times S^q)} \arrow{e,t}{\eta^{Diff}} \arrow{s,r}{F} \node{\prod_{i=1}^3\pi_{p_i}(G/O)} \arrow{s,r}{F} \arrow{e,t}{\theta^{Diff}} \node{L_{n}(e)} \arrow{s,r}{=}\\
\node{\prod_{i=1}^3\pi_{p_i+1}(G/Top)} \arrow{e} \node{L_{p+q+1}(e)} \arrow{e,t}{0} \node{\prod_{i=1}^2\pi_{p_i}(G/Top)} \arrow{e,t}{\eta^{Top}} \node{\prod_{i=1}^3\pi_{p_i}(G/Top)} \arrow{e,t}{\theta^{Top}} \node{L_{n}(e).}
\end{diagram}
\end{equation}
Here $\eta^{Top}(x, y) = (x, y, -xy)$ and $\theta^{Top}(x, y, z) = xy +z$ where the product $xy$ is defined by the pairing $\alpha(p, q) :L_p(e) \tensor L_q(e) \rightarrow L_{p+q}(e)$ as in \cite{Ra1}.  Note that $\alpha(p, q) = 0$ unless $(p, q) = (4j, 4k)$ when it is isomorphic to multiplication by $8$ and that our identification of $[S^p \times S^q, G/Top]$ differs from that of \cite{Ra4}.

We see that calculating $\mc{S}^{Diff}(S^p \times S^q)$ amounts to calculating the image of $\eta^{Diff}$ in $\prod_{i=1}^3\pi_{p_1}(G/O)$ and the action of $L_{p+q+1}(e)$ on $\mc{S}^{Diff}(S^p \times S^q)$.  For both points, it is helpful to recall that in general $\Theta_n$ acts on $\mathcal{S}^{Diff}(M)$ for any manifold $M$ and that by \cite{Bro} the action of $L_{n+1}(e)$ on $\mc{S}^{Diff}(M)$ factors through $\omega^{Diff} : L_{n+1}(e) \rightarrow bP_{n+1} \subset \Theta_n$ (see Section \ref{surgprelimsec}).  We begin stating our results with the illustrative special cases of $S^3 \times S^4$ and $S^4 \times S^4$ where we use the isomorphism $\pi_4(G/O) \cong \Z$.

\begin{Theorem} \label{S^3xS^4thm}
For $M = S^3 \times S^4$ there is a short exact sequence of pointed sets
\[ 0 \dlra{} \Z_{28} \dlra{\omega^{bP}} \mathcal{S}^{Diff}(S^3 \times S^4) \dlra{\eta} \Z \dlra{} 0 \]
where, $\Z_{28} \cong bP_8 = \Theta_7$ acts transitively on the fibres of $\eta$, $7 \cdot bP_8 \cong \Z_4$ acts freely on all of $\mathcal{S}^{Diff}(S^3 \times S^4)$ but $4 \cdot bP_8 \cong \Z_7$ acts freely on $\eta^{-1}(v)$
if and only if $v \in \Z$ is divisible by $7$.
\end{Theorem}

%
%

\begin{Corollary} \label{S^3xS^4cor}
The sets $\mathcal{S}^{Diff}(S^3 \times S^4)$ and $\mathcal{N}^{Diff}(S^3 \times S^4) \equiv \Z$ cannot be given group structures such that $\eta^{Diff}$ is a homomorphism.  The same holds for $\mathcal{S}^{Diff}(S^3 \times S^4)$ and both $\omega^{bP} : bP_8 \rightarrow \mc{S}^{Diff}(S^3 \times S^4)$ and $\omega^{} : L_8(e) \rightarrow \mathcal{S}^{Diff}(S^3 \times S^4)$.  
\end{Corollary}

\begin{proof}
If $\eta$ were a homomorphism then its fibres would all be cosets of the kernel and hence have the same size.  If $\omega^{bP}$ (resp. $\omega$) were a homomorphism, the stabilisers of the associated action would all be equal to $(\omega^{bP})^{-1}([{\rm Id}])$ (resp. $\omega^{-1}([{\rm Id}])$).
\end{proof}

\begin{Theorem} \label{S^4xS^4thm}
For $M = S^4 \times S^4$ there is a short exact sequence of pointed sets
\[ 0 \dlra{} \Z_2 \dlra{\omega^\Theta} \mathcal{S}^{Diff}(S^4 \times S^4) \dlra{i^* \circ \eta} \Z \times \Z \dlra{\del} \Z_7 \dlra{} 0 \]
where $\Z_2 \cong \Theta_8$ acts freely and transitively on the fibres of $i^* \circ \eta$ and $\del(u, v) = u v$ mod $7$.
\end{Theorem}

\begin{Remark}
In Subsection \ref{grstrucsubsec} below we use Theorem \ref{S^4xS^4thm} to show that the image of the forgetful map $\mc{S}^{Diff}(S^4 \times S^4) \rightarrow \mc{S}^{Top}(S^4 \times S^4)$ is not a subgroup.
\end{Remark}

The examples of $S^3 \times S^4$ and $S^4 \times S^4$ show that in the smooth case the map
\[ i^* \circ \eta^{Diff} : \mc{S}^{Diff}(S^p \times S^q) \dlra{} \left( \mathcal{N}^{Diff}(S^p \times S^q) \equiv \prod_{i=1}^3 \pi_{p_i}(G/O) \right) \dlra{} \prod_{i=1}^2 \pi_{p_i}(G/O) \]
is in general neither injective nor surjective.  We place it into an exact sequence.

\begin{Theorem} \label{mainthm1}
Define the integer $t_i$ by $t_{4k} = |{\rm Cok}(\pi_{4k}(G/O) \dlra{F} \pi_{4k}(G/Top))|$ so that $t_4 = 2$ and $t_{4k} = |bP_{4k}|$ if $k > 1$ and by $t_i = 0$ if $i \neq 0$ mod $4$.  For $p, q \geq 2$ and $p+q \geq 5$, there is an exact sequence of pointed sets
%
%
%
\[ 0 \dlra{} \Theta_{p+q} \dlra{\omega^\Theta} \mathcal{S}^{Diff}(S^p \times S^q) \dlra{i^* \circ \eta^{Diff}} \pi_p(G/O) \times \pi_q(G/O) \dlra{\del} 8 \, t_p \, t_q \cdot bP_{p+q} \dlra{} 0\]
where $\Theta_{p+q}$ acts transitively on the fibres of $i^* \circ \eta^{Diff}$ and the map $\del$ is the composite
\[\prod_{i=1}^2 \pi_{p_i}(G/O) \dlra{} \prod_{i=1}^2 \pi_{p_i}(G/Top) \cong L_p(e) \times L_q(e) \dlra{\alpha(p, q)} L_{p+q}(e) \dlra{\omega^{Diff}} bP_{p+q}. \]
\end{Theorem}

Next we describe the action of $\Theta_{p+q}$ on $\mathcal{S}^{Diff}(S^p \times S^q)$ which requires some further preliminaries.   Firstly recall the contracted Kervaire-Milnor sequence extracted from $(\ref{seseqnS^n} )$
\[ 0 \dlra{} bP_{n+1} \dlra{} \Theta_{n} \dlra{} \pi_n(G/O) \dlra{\theta^{Diff}} \L_n(e).\]
Hence $\Theta_n/bP_{n+1} \cong (\theta^{Diff})^{-1}(0) \subset \pi_n(G/O)$.  When $n=4k$ we choose a splitting $\tau_{4k} \times \phi_{4k} : \pi_{4k}(G/O) \cong \Theta_{4k} \times \Z$ and when $p+q = 4k$ we use this splitting and that of $( \ref{nmlspliteqn} )$ to define
\[ \bar \eta^{} = (\tau_{4k} \times i^*) \circ \eta^{Diff} : \mc{S}^{Diff}(S^p \times S^q) \rightarrow \Theta_{4k} \times \pi_p(G/O) \times \pi_q(G/O). \]
If $p+q \neq 4k$ we set $\bar \eta^{} = \eta^{Diff}$.  Secondly, given the symmetry in $(p, q)$ we adopt the convention that if $p + q$ is odd then $q$ is assumed even.  If $(p, q) = (4j-1, 4k)$ let $i_{4k} : S^{4k} \hra S^{4j-1} \times S^{4k}$ be the standard inclusion and define the surjection
\[
d: \mathcal{S}^{Diff}(S^{4j-1} \times S^{4k}) \sra \Z, ~~[N, f] \mapsto d([N, f]) :=  \phi_{4k}(i_{4k}^*(\eta^{Diff}([N, f]))). \]

\begin{Theorem} \label{mainthm2}
For $p, q \geq 2$ and $p+q \geq 5$, the action of $\Theta_{p+q}$ on $\mathcal{S}^{Diff}(S^p \times S^q)$ is free unless $p=4j-1$ and $q = 4k$ in which case the stabilisers of the action are all subgroups of $bP_{4(j+k)}$ of odd order.  There is a long exact sequence of pointed sets
\[ 0 \rightarrow{} bP_{p+q+1} \dlra{\omega^{bP}} \mathcal{S}^{Diff}(S^p \times S^q) \dlra{\bar \eta} \left(\Theta_{p+q}/bP_{p+q+1} \times \prod_{i=1}^2\pi_{p_i} (G/O) \right) \dlra{0 \times \del} 8\,t_p \, t_q \cdot bP_{p+q} \rightarrow{} 0\]
where $bP_{p+q+1}$ acts transitively on the fibres of $\bar \eta$ and for all $[N, f] \in \mathcal{S}^{Diff}(S^{4j-1} \times S^{4k})$ the stabilisers of the action of $bP_{4(j+k)}$ are given by
\[(bP_{4(j+k)})_{[N, f]} = 8 \, d([N, f]) \, t_{4j}\, t_{4k} \cdot bP_{4(j+k)}.\]  
For example: $8 \,t_4\, t_4 \cdot bP_8 \cong \Z_7,~ 8 \, t_4 \, t_8 \cdot bP_{12} \cong \Z_{31}, ~8 \, t_4 \, t_{12} \cdot bP_{12} \cong \Z_{127}$ and $8 \, t_8 \, t_8 \cdot bP_{16} \cong \Z_{127}$.
\end{Theorem}

\begin{Remark} \label{oddrem}
In Section \ref{numsec} we review the formulae for $t_{4i}$ and show that $8 \, t_{4j} \, t_{4k} \cdot bP_{4(j+k)}$ is always of odd order.  It seems likely that is never or very rarely trivial. 
\end{Remark}

\begin{Remark}
Theorems \ref{S^3xS^4thm}, \ref{S^4xS^4thm}, \ref{mainthm1} and \ref{mainthm2} have been attained by a process of trial and error: Example 13.26 of \cite{Ra3} describes $\mathcal{S}^{Diff}(S^p \times S^q)$ incorrectly and in \cite{C1} I stated incorrectly that the action of $bP_8$ on $\eta^{-1}(\pm 1) \subset \mc{S}^{Diff}(S^3 \times S^4)$ is trivial.
\end{Remark}
The rest of this paper is organised as follows: we complete the introduction by recalling some of the background concerning group structures on $\mc{S}^{Cat}(M)$ and by spelling out the implications of Theorems \ref{S^3xS^4thm} and \ref{mainthm1} in this context.   In Section \ref{explsec} we prove Theorems \ref{S^3xS^4thm} and \ref{S^4xS^4thm} and give more or less explicit descriptions of all the elements of $\mc{S}^{Diff}(S^3 \times S^4)$ and $\mc{S}^{Diff}(S^4 \times S^4)$.  We also classify smooth manifolds homotopy equivalent to $S^3 \times S^4$ or $S^4 \times S^4$: the classifications are not achieved by computing the action of homotopy self-equivalences on the structure set, a problem we leave aside in this paper, but by applying classification theorems of Wilkens and Wall \cite{Wi, Wa1} as well as an unpublished theorem of our own about the inertia group of $3$-connected $8$-manifolds \cite{C2}.  In Section \ref{surgprelimsec} we recall and develop some general results on the surgery exact sequence and in particular the action of $L_{n+1}(e)$ on $\mc{S}^{Diff}(M)$.  Lemma \ref{keysurgoblem} and Corollary \ref{stabhomeodepcor} may be of independent interest in this regard.  In Section \ref{proofsec} we apply the results of Section \ref{surgprelimsec} to prove Theorems \ref{mainthm1} and \ref{mainthm2}.  Finally in Section \ref{numsec} we give a summary of the calculation of the order of $bP_{4k}$.

\subsection{Group structures on $\mathcal{S}^{Cat}(M)$} \label{grstrucsubsec}
For $Cat$ = $Top$ or $Diff$ the operation of Whitney sum of stable bundles gives $G/Cat$ ($G/Diff = G/O$) the structure of an infinite loop space so that $\mc{N}^{Cat}(M) \equiv [M, G/Cat]$ admits the structure of an abelian group: for example it follows from Lemma \ref{grpprodlem} that our identification $[S^p \times S^q, G/Cat] \equiv \prod_{i=1}^3\pi_{p_i}(G/Cat)$ is a group isomorphism when the later product is a product of groups.  Moreover, the surgery exact sequences of $(\ref{seseqn1})$ are exact sequences of groups to the left of $
\mc{S}^{Cat}(M)$.  But, as is well known and as we saw above in $( \ref{seseqnS^pxS^q})$, the surgery obstruction maps $\theta^{Cat} : [M ,G/Cat] \rightarrow L_n(e)$ are not homomorphisms when $[M, G/Cat]$ has the Whitney sum group structure.  

In addition to the Whitney sum infinite loop space structure $G/Top$ has another infinite loop space structure derived from its identification as the initial space of the $L$-theory spectrum \cite{Q}.  It is a theorem of Siebenmann \cite{Si} that with this group structure $\theta^{Top}$ is a homomorphism and that $\mathcal{S}^{Top}(M)$ admits a group structure making the topological surgery exact sequence into a long exact sequence of abelian groups.  Ranicki (see \cite{Ra3}) later constructed an algebraic surgery exact sequence of abelian groups for a topological manifold $M$ and an isomorphism to the topological surgery exact sequence for $M$.  

In \cite{N} Nicas asked if the smooth surgery exact sequence might be a long exact sequence of groups.  In \cite{We} Weinberger showed that the forgetful map $F: \mathcal{S}^{Diff}(M) \rightarrow \mathcal{S}^{Top}(M)$ is finite-to-one with image containing a subgroup of finite index and in a remark left open whether the image of $F$ is in general a subgroup of $\mathcal{S}^{Top}(M)$.  It is now well known to the experts that neither of these possibilities occurs in general (see for example \cite{Ra3}[\S 13.3]) but no explicit examples have appeared in the literature.   Corollary \ref{S^3xS^4cor} provides a negative answer to Nicas' question and Corollary \ref{imFnotgrpcor} below settles the matter for ${\rm Im}(F)$.  Before proceeding to Corollary \ref{imFnotgrpcor} we present another consequence of Theorem \ref{S^3xS^4thm}.  The forgetful map $F$ fits into a short exact sequence of pointed sets
\[ [M, Top/O] \dlra{} \mathcal{S}^{Diff}(M) \dlra{F} \mathcal{S}^{Top}(M) \]
where by \cite{H-M} and \cite{K-S} $[M, Top/O]$ may be identified with the set of concordance classes of smooth structures on $M$. 

\begin{Corollary}
The fibres of $F_{}$ are not in general equivalent sets. 
\end{Corollary}

\begin{proof}
We take $M = S^3 \times S^4$ where the topological normal invariant map gives an identification $\eta^{Top}: \mathcal{S}^{Top}(S^3 \times S^4) \cong \pi_4(G/Top)
\cong \Z$ 
such that $\eta^{Top} \circ F_{} = F \circ \eta^{Diff}$ where $F: \pi_4(G/O) \rightarrow \pi_4(G/Top)$ is the inclusion onto the subgroup of index two by Rochlin's Theorem \cite{Ro}.  Hence by Theorem \ref{S^3xS^4thm},
\[ F_{}^{-1}(2y) \equiv 
\left\{ \begin{array}{cl} \Z_4 & \text{if $7$ is prime to $y$,} \\ \Z_{28} & \text{if $7$ divides $y$.} 
\end{array} \right.\]
\end{proof}

\begin{Corollary} \label{imFnotgrpcor}
The image of the forgetful map $F_{}:  \mc{S}^{Diff}(S^{4j} \times S^{4k}) \rightarrow \mc{S}^{Top}(S^{4j} \times S^{4k})$ is not a subgroup for $S^4 \times S^4, S^4 \times S^8, S^4 \times S^{12}$ and $S^8 \times S^8$.
\end{Corollary}

\begin{proof}
By \cite{Ra2, K-L}[Ex. 20.4, \S 7] the bijection $i^* \circ \eta^{Top} : \mc{S}^{Top}(S^p \times S^q) \equiv L_p(e) \times L_q(e)$ of $( \ref{topSS^pxS^qeqn} )$ is in fact an isomorphism of groups for the Siebenmann group structure on $ \mc{S}^{Top}(S^p \times S^q)$.  The map $F$ factors as 
\[\mc{S}^{Diff}(S^{4j} \times S^{4k}) \dlra{i^* \circ \eta^{Diff}} \pi_{4j}(G/O) \times \pi_{4k}(G/O) \rightarrow{} \pi_{4j}(G/Top) \times \pi_{4k}(G/Top) \cong \mc{S}^{Top}(S^{4j} \times S^{4k}) \]
and by definition the image of $\pi_{4i}(G/O) \rightarrow \pi_{4i}(G/Top)$ is $t_{4i} \cdot \pi_{4i}(G/Top)$.  Applying Theorem \ref{mainthm1} we obtain the equality
\[ {\rm Im}(F) = \{ (t_{4j}\, x, t_{4k} \, y) \,| \, t_{4j} \,t_{4k} \, xy \in t_{4(j+k)}\cdot L_{4(j+k)}(e) \} \subset  L_{4j}(e) \times L_{4k}(e)\]
so that ${\rm Im}(F)$ is a subgroup if and only if $8 \, t_{4j} \,t_{4k} \cdot bP_{4(j+k)} = 0$.  The calculations in Section \ref{numsec} show that this does not occur for the dimensions listed (and probably for very few pairs, or indeed no pairs, $(4j, 4k)$).
\end{proof}

We finish the introduction with some remarks related to groups structures on $\mc{S}^{Diff}(M)$.

\begin{Remark}
For a closed, smooth $3$-dimensional manifold $M^3$ and an appropriate definition of the smooth structure set, $\mathcal{S}^{Diff}(M^3)$, Kro \cite{K} proves that $\mathcal{S}^{Diff}(M^3)$ admits a group structure making the smooth surgery exact sequence into a long exact sequence of groups.
\end{Remark}

\begin{Remark}
Since $8 \, t_{4j} \, t_{4k} \cdot bP_{4(j+k)}$ is always of odd order our results do not obstruct a group structure on a ``$2$-local'' smooth surgery exact sequence, whatever that might be precisely.
\end{Remark}

With regard to placing group structures on exact sequences of pointed sets, the reader may wish to verify the following 
\begin{Lemma} \label{grpposslem}
Let $A \dlra{f} B \dlra{g} C$ be an exact sequence of pointed sets where $A$ and $C$ are also groups such that $A$ acts on $B$ with orbits the fibres of $g$.  Then $B$ can be given a group structure such that $f$ and $g$ are homomorphisms if and only if
\begin{enumerate}
\item the image of $g$ is a subgroup of $C$ and
\item the stabilisers of the action of $A$ on $B$ are all equal to a fixed normal subgroup of $A$.
\end{enumerate}
\end{Lemma}
\noindent
From Lemma \ref{grpposslem} and Theorem \ref{mainthm2} we see that  $\mc{S}^{Diff}(S^p \times S^q)$ admits a groups structure such that 
\[ L_{p+q+1}(e) \dlra{\omega^{Diff}} \mc{S}^{Diff}(S^p \times S^q) \dlra{\eta^{Diff}} \mc{N}^{Diff}(S^p \times S^q) \]
is an exact sequence of groups if and only if $(p, q) \neq (4j-1, 4k)$ or $(4j, 4k)$.  Assuming $p\leq q$, we hypothesise that some form of framed connected sum over $S^p$ can be used to give a geometrically define group structure on $\mc{S}^{Diff}(S^p \times S^q)$ which would be closely related to the isotopy group of $S^p \times S^{q-1}$.  
\\ \\
\noindent
{\bf Acknowledgments:}  
I would like to thank Matthias Kreck, Andrew Ranicki, Jim Davis, Kent Orr and  Wolfgang L\"{u}ck for several helpful suggestions and questions.

\section{The structure sets of $S^3 \times S^4$ and $S^4 \times S^4$} 
\label{explsec}
In this section we prove Theorems \ref{S^3xS^4thm} and \ref{S^4xS^4thm}, give representatives for every element of $\mc{S}^{Diff}(S^3 \times S^4)$ and $\mc{S}^{Diff}(S^4 \times S^4)$ and classify smooth manifolds homotopy equivalent to $S^3 \times S^4$ or $S^4 \times S^4$.  

Theorem \ref{S^3xS^4thm} follows immediately from Theorem \ref{mainthm2} and the isomorphisms $\pi_3(G/O) \cong \pi_7(G/O) \cong 0$, $\pi_4(G/O) \cong \Z$ and $bP_8 \cong \Z_{28}$.  It gives the short exact sequence of pointed sets
\begin{equation} \label{ses34eqn}
0 \dlra{} \Z_{28} \dlra{\omega^{bP}} \mc{S}^{Diff}(S^3 \times S^4) \dlra{ \eta } \Z \dlra{} 0 
\end{equation}
where $bP_8$ acts transitively on the fibres of $\eta$ and the stabiliser of $[N, f] \in \mc{S}^{Diff}(S^3 \times S^4)$ is $32 \eta([N, f]) \cdot bP_8$.  We obtain a section for $\eta$ by recalling from \cite{C-E} that for each $v \in \Z$ there are fibre homotopy equivalences $f_v  : N_{v} \rightarrow S^3 \times S^4$ such that $\eta([N_v, f_v]) = v$ where $N_v$ is the total space of the $3$-sphere bundle over $S^4$ with trivial Euler class and first Pontrjagin class $48v \subset \Z = H^4(S^4)$: here and below all (co)homology groups are with integer coefficients.  (Even though \cite{C-E} considers bundles with non-trivial Euler number nothing hangs on this assumption).  For precision we prove that the map $v \mapsto [N_v, f_v]$ gives a well-defined section of $\eta$.

\begin{Lemma} \label{etasec34lem}
The structure invariant $[N_v, f_v]$ is independent of the choice of fibre homotopy equivalence $f_v$.
\end{Lemma}

\begin{proof}
Up to fibre homotopy any two such maps, $f_v$ and $f'_v$, differ by an element of $\pi_4(SG(4))$ where $SG(4)$ is the monoid of orientation preserving self-homotopy equivalences of $S^3$.  Standard arguments show that the forgetful map $\pi_4(SO(4)) \rightarrow \pi_4(SG(4))$ is an isomorphism and that every element of $\pi_4(SO(4))$ can be realised by a fibrewise diffeomorphism of $N_v \rightarrow S^4$.
\end{proof}


\begin{Theorem} \label{class34thm}
\begin{enumerate}
\item
The structure invariants $[\Sigma \sharp N_v, f_v]$ exhaust $\mc{S}^{Diff}(S^3 \times S^4)$ as $v$ ranges over $\Z$ and $[\Sigma]$ ranges over $bP_8$.
\item We have $[\Sigma \sharp N_{v_0}, f_{v_0}] = [\Sigma \sharp N_{v_0}, f_{v_1}]$ if and only if $v_0 = v_1$ and $[\Sigma_0] - [\Sigma_1] \in 32 \, v_0 \cdot bP_8$.  
\item
There is a diffeomorphism $\Sigma_0 \sharp N_{v_0} \cong \Sigma_1 \sharp N_{v_1}$ if and only if $v_0 = \pm v_1$ and $[\Sigma_0] - [\Sigma_1] \in 2 \, v_0 \cdot bP_8$.
\end{enumerate}
\end{Theorem}

\begin{proof}
The first two statements follow from the fact that $\eta([N_v, f_v]) = v$, Lemma \ref{etasec34lem} and Theorem \ref{S^3xS^4thm}.  For the third, we observed in \cite{C-E} that there are bundle isomorphisms covering the antipodal map on $S^4$ which give diffeomorphisms $N_v \cong N_{-v}$ so it remains to compute the inertia group of $N_v$, $I(N_v) := \{ [\Sigma] \in \Theta_7 = bP_8 \, | \, \Sigma \sharp N_v \cong N_v \}$.   A theorem of Wilkens, \cite{Wi}[Theorem 1 (ii)], asserts in part that $I(N_{v}) \cong \Z_{k(v)} \subset bP_8$ where $k(v) = 14/(14, v)$ and $(14, v)$ is the g.c.d. of $14$ and $v$: this completes the proof.
\end{proof}

We now consider $\mc{S}^{Diff}(S^4 \times S^4)$.  Theorem \ref{S^4xS^4thm} follows immediately from Theorem \ref{mainthm2} and the isomorphisms $bP_8 \cong \Z_{28}$, $\pi_4(G/O) \cong \Z$, and $\pi_8(G/O) \cong \Theta_8 \cong \Z_2$.  It gives the exact sequence of pointed sets
\begin{equation} \label{ses44eqn}
0 \dlra{} \Z_2 \dlra{\omega^\Theta} \mc{S}^{Diff}(S^4 \times S^4) \dlra{i^* \circ \eta} \Z \times \Z \dlra {\del} \Z_7 \dlra{} 0
\end{equation}
where $\Z_2 \cong \Theta_8$ acts freely and transitively on the fibres of $i^* \circ \eta$ and $\del(u, v) = uv$ mod $7$.  We shall construct homotopy equivalences $f_{u, v} : N_{u, v} \rightarrow S^4 \times S^4$ with $(i^* \circ \eta)([N_{u, v}, f_{u, v}]) = (u, v)$.  Let $W_u$ be the $4$-disc bundle which cobounds $N_u$ above: that is $W_u \rightarrow S^4$ is a linear $4$-disc bundle with trivial Euler class and first Pontrjagin class $48u \in H^4(S^4)$.  If $D^4 \subset S^4$ is an embedded disc in the base space then the bundle $W_u \rightarrow S^4$ is trivial over $D^4$ and there are fibre homotopy equivalences $g_u : W_u \simeq D^4 \times S^4$ which we may assume are equal to the identity over $D^4 \times D^4$.  We let $W_{u, v}$ be the manifold obtained by plumbing $W_u$ and $W_v$ together: i.e. we identify $D^4 \times D^4 \subset W_u$ with $D^4 \times D^4 \subset W_v$ by exchanging fibre and base space coordinates.  We thus obtain homotopy equivalences $f'_{u, v} : W_{u, v} \rightarrow (S^4 \times S^4) - D^8$ since $(S^4 \times S^4) - D^8$ is the manifold obtained by plumbing two trivial bundles together.  

Manifolds of the type $W = W_{u, v}$ were classified by Wall in \cite{Wa1} using their intersection form, $\lambda_W: H_4(W) \times H_4(W) \rightarrow \Z$ and stable tangential invariant $S\alpha(W) : H_4(W) \rightarrow \pi_3(SO) \cong \Z$.  In the case of $W_{}$, standard arguments show that these data are given by:
\[H_4(W_{}) = \Z^2(x, y), ~~\lambda_{W_{}} = \left( \begin{array}{cc}0 & 1 \\ 1 & 0 \end{array} \right), ~~S\alpha(W)(x) = 24u, ~~S\alpha(W)(y) = 24v .\]
The boundary of $W_{}$, $\del W_{}$, is a homotopy sphere whose diffeomorphism class is determined by the following formula (see \cite{Wa1}[Theorem 4] and \cite{E-K}):
\[ \mu(W) = (\sigma(W)-(S\alpha(W)^2)/8 \cdot 28 ~\in ~\Z_{28} \subset \Q/\Z, \]
where $\sigma(W)$ is the signature of $W$, we regard $S\alpha(W)$ as an element of $H^4(W)$ and $S\alpha(W)^2$ is calculated using the isomorphism $H^4(W) \cong H^4(W, \del W)$ and then evaluating on the fundamental class of $W$.  Clearly $\sigma(W) = 0$ and one may check that
\[\mu(W_{u, v}) = - 2 \cdot 24 \cdot 24 \cdot uv / 28 \in \Q/ \Z.\] 
Hence $\del W_{u, v}$ is diffeomorphic to $S^7$ if and only if $7$ divides $uv$ which confirms $(\ref{ses44eqn})$.  When $\del W_{u, v} \cong S^7$ we let $N_{u, v, \phi} : = W_{u, v} \cup_\phi (-D^8)$ where $\phi : \del W_{u, v} \cong S^7 = \del D^8$ is a diffeomorphism.  Extending $f'_{u, v}$ by coning we obtain $f_{u, v} : N_{u, v, \phi} \rightarrow S^4 \times S^4$ and it is easy to check that $(i^* \circ \eta)([N_{u, v, \phi}, f_{u, v}]) = (u, v)$.  As for the action of $\Theta_8$, the choices for $\phi : \del W_{u, v} \cong S^7$ are parametrised up to isotopy by $\Theta_8 \cong \pi_0({\rm Diff}(S^7))$ and so by $( \ref{ses44eqn} )$ we obtain all of $\mc{S}^{Diff}(S^4 \times S^4)$ by varying $u, v$ and $\phi$.  

In the following theorem recall that closed, smooth manifolds $N_0$ and $N_1$ are said to be {\em almost diffeomorphic} if there is a homotopy sphere $\Sigma$ such that $N_0$ and $\Sigma \sharp N_1$ are diffeomorphic.
\begin{Theorem} \label{class44thm} 
\begin{enumerate}
\item
The structure invariants $[N_{u, v, \phi}, f_{u, v}]$ exhaust $\mc{S}^{Diff}(S^4 \times S^4)$.
\item
For $i = 0, 1$ and two pairs $(u_i, v_i) \in \Z \times \Z$ with $7$ dividing $u_i \, v_i$, two manifolds $N_{u_0, v_0, \phi_0}$ and $N_{u_1, v_1, \phi_1}$ are almost diffeomorphic if and only if for some $\epsilon \in \{ \pm 1\}$ the unordered pairs $\{u_0, v_0 \} $ and $\{\epsilon u_1, \epsilon v_1\}$ are equal.  
\item
The inertia group of every $N_{u, v, \phi}$ is trivial so that there are two diffeomorphism types within each almost diffeomorphism type of manifolds homotopy equivalent to $S^4 \times S^4$.
\end{enumerate}
\end{Theorem}
\begin{proof}
It remains to prove the second and third statements.  By \cite{Wa1} $3$-connected $8$-manifolds are classified up to almost diffeomorphism by their intersection form and stable tangential invariant.  Thus we must classify triples $(H, \lambda, S\alpha)$ where $H \cong \Z^2(x, y)$, $\lambda$ is the hyperbolic form on $H$ and $S\alpha : H \rightarrow \Z$ is a homomorphism.  In this case $S\alpha(W_{u,v})(x) = 24u$ and $S\alpha(W_{u, v})(y) = 24v$.  The group of automorphisms of $\lambda$ is isomorphic to $\Z_2(T) \times \Z_2(-{\rm Id})$ where $T(x, y) = (y, x)$ and $(-{\rm Id})(x, y) = (-x, -y)$ and the almost diffeomorphism classification follows immediately.

The third statement follows from the fact that ${\rm Im}(S\alpha(W_{u, v})) \subset 24 \cdot \Z$ and the following theorem which is proven in \cite{C2}.

\begin{Theorem}
Let $N$ be a closed, smooth $3$-connected $8$ manifold with stable tangential invariant $S\alpha : H_4(N) \rightarrow \pi_3(SO) = \Z$.  The inertia group of $N$ is trivial if and only if ${\rm Im}(S\alpha) \subset 4 \cdot \Z$.
\end{Theorem}
\end{proof}

\section{Surgery Preliminaries} \label{surgprelimsec}
In this section we quickly recall some definitions and then record preliminary results about the surgery exact sequence.  The latter concern product manifolds $M = M_0 \times M_1$, the action of the smooth surgery exact sequence for $S^n$ on the surgery exact sequence of general $M$ and the action of $L_{n+1}(e)$ on $\mc{S}^{Diff}(M)$.  Throughout $M$ is a closed $Cat$ manifold, $Cat$ = $Diff$ or $Top$, and we assume that $M$ is simply-connected for simplicity.  We point out that appropriate versions of the results in this section should hold for any fundamental group: in particular Lemma \ref{keysurgoblem} and Corollary \ref{stabhomeodepcor} may be of interest in other contexts.

Let $Y$ and $W$ be compact $Cat$ manifolds with possibly empty boundaries $\del W$ and $\del Y$.  A degree one normal map $(f, b) : Y \rightarrow W$ is a degree one map $f : Y \rightarrow W$ together with a stable bundle isomorphism $b: \nu_Y \rightarrow \xi$ where $\nu_Y$ is the stable normal bundle of $Y$, $b$ covers $f$ and $\xi$ is some stable vector bundle over $W$ (necessarily fibre homotopy equivalent to $\nu_W$).  If $X \subset \del W$ is a nicely embedded codimension-$0$ submanifold of the boundary of $W$ then $(f, b)$ is called ``rel.\,\,$X$'' if $f|_{f^{-1}(X)} : f^{-1}(X) \rightarrow X$ is a $Cat$ isomorphism.  A normal bordism of rel.\,\,$\del W$ degree one normal maps $(Y_i, f_i, b_i)$, $i= 0, 1$, is a degree one normal map $(g, c) : Z \rightarrow W \times [0, 1]$ rel.\,\,$\del W \times [0, 1]$ restricting to $(Y_i, f_i, b_i)$ over $W \times \{ i \}$.  The set of normal bordism classes of rel.\,\,$\del W$ degree one normal maps is denoted $\mc{N}^{Cat}(W)$.  Here we use the convention that if the boundary is not mentioned explicitly then surgery problems are assumed solved on, and relative to, the boundary.

When $M$ is a closed $Cat$ manifold the map $\eta^{Cat} : \mc{S}^{Cat}(M) \rightarrow \mc{N}^{Cat}(M)$ is defined by mapping $[N, f]$ to $[(f, b) : N \rightarrow M]$ where $b : \nu_N \cong f^{-1*} (\nu_N)$ is the canonical bundle map and $\xi = f^{-1*}(\nu_N)$.  The map $\eta^{Cat}$ fits into the long exact sequence from $(\ref{seseqn1})$
\[ \dots \dlra{} \mc{N}^{Cat}(M \times [0,1]) \dlra{\theta^{Cat}} L_{n+1}(e) \dlra{\omega^{Cat}} \mc{S}^{Cat}(M) \dlra{\eta^{Cat}} \mc{N}^{Cat}(M) \dlra{\theta^{Cat}} L_{n}(e).\]
We conclude this brief review of the surgery exact sequence here and refer the reader to \cite{Bro, Wa2, Lu, Ra2} for definitions of the groups $L_n(e)$, the surgery obstruction maps $\theta^{Cat} : \mc{N}^{Cat}(M) \rightarrow L_n(e)$ and $\theta^{Cat} : \mc{N}^{Cat}(M \times [0, 1]) \rightarrow L_{n+1}(e)$ and the action map $\omega : L_{n+1}(e) \rightarrow \mc{S}^{Cat}(M)$.  As a final point, we remark that all of the above definitions make sense in any dimension but it is only in dimensions $n \geq 5$ that the surgery exact sequence is guaranteed to be exact.

Now suppose that $M = M_0 \times M_1$ is the produced of two closed manifolds.  The product of two degree one normal maps is again a degree one normal map and we obtain a map 
\[ \pi : \mc{N}^{Cat}(M_0) \times \mc{N}^{Cat}(M_1) \rightarrow \mc{N}^{Cat}(M_0 \times M_1).\]
We first recall how the surgery obstruction map $\theta_{M_0 \times M_1}$ restricted to the image of $\pi$ is related to $\theta_{M_0}$ and $\theta_{M_1}$.  This requires some facts from \cite{Ra1}[\S I.8]: there are $4$-periodic symmetric $L$-groups $L^i(e)$ with $L^i(e) \cong \Z, \Z_2, 0, 0$ as $i = 0, 1, 2, 3$ mod $4$ such that any closed $n$-dimensional manifold $M$ has a symmetric signature $\sigma^*(M) \in L^n(e)$.  Moreover, there are homomorphisms $L_i(e) \rightarrow L^i(e)$, necessarily zero unless $i=4k$ when $L_{4k}(e) \rightarrow L^{4k}(e)$ is multiplication by $8$ and there are product maps $L_i(e) \tensor L^j(e) \rightarrow L_{i+j}(e)$ such that the induced product $L_i(e) \tensor L_j(e) \rightarrow L_{i+j}(e)$ is zero unless $(i, j) = (4k, 4l)$ when it is isomorphic to multiplication by $8$.  The following is proven as part of \cite{Ra1}[\S II.8.1]:
\begin{Proposition} \label{surgprodprop}
Let $M_0$ and $M_1$ be closed manifolds of dimensions $n_0$ and $n_1$ and let $(x, y) \in \mc{N}^{Cat}(M_0) \times \mc{N}^{Cat}(M_1)$.  Then
\begin{equation} \label{surgobeqn}
 \theta_{M_0 \times M_1}(\pi(x, y)) = \theta_{M_0}(x) \theta_{M_1}(y) +  \theta_{M_1}(y)\sigma^*(M_0) + \theta_{M_0}(x)\sigma^*(M_1) ~\in~ L_{n_0 + n_1}(e).
 \end{equation}
 %
 \end{Proposition}

Identifying $\mc{N}^{Cat}(M_i) = [M_i, G/Cat]$ for $i = 0, 1$ we also obtain a map 
\[\pi : [M_0, G/Cat] \times [M_1, G/Cat] \rightarrow [M_0 \times M_1, G/Cat].\]
%

\begin{Lemma} \label{grpprodlem}
The map $\pi : [M_0, G/Cat] \times [M_1, G/Cat] \rightarrow [M_0 \times M_1, G/Cat]$ is a homomorphism with respect to the Whitney sum infinite loop structure on $G/Cat$.
\end{Lemma}

\begin{proof}
We give only a sketch.  The identification $\mc{N}^{Cat}(M) = [M, G/Cat]$, see  \cite{M-M}[Theorem 2.23] for $Cat = Diff$ (and also $PL$), runs as follows: a map $M \rightarrow G/Cat$ defines a fibre homotopy trivialisation of vector bundles over $M$ and making this transverse to the zero section we obtain a degree one normal map to $M$.  The Whitney sum multiplication on $G/Cat$ takes a pair of fibre homotopy trivialisations of vector bundles to the Whitney sum of these trivialisations. The product of transverse maps is again transverse so now we reduce to a simple point-set identity for the inverse image of certain product maps which completes the proof.

The same argument also works for topological manifolds given topological transversality \cite{F-Q, K-S}. 
%

%
%
\end{proof}

\begin{Remark}
We shall use Proposition \ref{surgprodprop} and Lemma \ref{grpprodlem} for $M_0 = S^p$ or $S^p \times S^1$ and $M_1 = S^q$ where $p, q \geq 2$ and so use the identification $\mc{N}^{Cat}(S^p) \equiv \pi_q(G/Cat)$ even in low dimensions.  The reader may object for $Cat = Top$, in particular $p=4$, that we are using the very deep mathematics of topological transversality for a comparatively simple outcome: calculating $\mc{S}^{Diff}(S^p \times S^q)$.  To this we reply firstly that our proof of Theorem \ref{mainthm1} uses only the smooth category and that our proof of Theorem \ref{mainthm2} can be easily modified so that the same is also true of it.  Our methods thus make no essential use of the topological category and so are comparatively elementary.  Secondly, an important aim of this paper is to compare the smooth and topological surgery exact sequences and in that respect we are happy to recall and use the full scope of topological surgery.
\end{Remark}

We now turn to preliminaries concerning the action of the smooth surgery exact sequence for $S^n$ on the smooth surgery exact sequence for an arbitrary smooth manifold $M$.  Firstly, recall that the group of homotopy $n$-spheres, $\mc{S}^{Diff}(S^n) = \Theta_n$, acts on the smooth structure set of $M$ by
\[ \Theta_n \times \mc{S}^{Diff}(M) \dlra{} \mc{S}^{Diff}(M), ~~~([\Sigma], [N, f]) \longmapsto [\Sigma \sharp N, f]\]
where we regard the connected sum $\Sigma \sharp N$ as a smooth manifold with the same underlying topological space as $N$ and with smooth structure differing from that of $N$ only on an $n$-disc.  There is also an action of the group of smooth normal invariants of $S^n$ on $\mc{N}^{Diff}(M)$,
\[ \mc{N}^{Diff}(S^n) \times \mc{N}^{Diff}(M) \dlra{} \mc{N}^{Diff}(M), ~~~([W, g, c], [N, f, b]) \longmapsto (W \sharp N, g \sharp f, c \sharp b) \]
where we take the connected sum in domain and range of degree one normal maps assumed to be the identity about the locus of the connected sum.

\begin{Lemma} \label{sescomlem}
The actions of $\Theta_n$ on $\mc{S}^{Diff}(M)$ and $\mc{N}^{Diff}(S^n)$ on $\mc{N}^{Diff}(M)$ are compatible so that there is a commuting diagram on long exact sequences
\begin{equation} \label{sesS^nactlem}
\begin{diagram}
\node{\dots} \arrow{e} \node{L_{n+1}(e)} \arrow{e,t}{\omega^{}_{S^n}} \arrow{s,r}{=} \node{\Theta_n} \arrow{e,t}{\eta^{}_{S^n}} \arrow{s,r}{} \node{\mathcal{N}^{Diff}(S^n)} \arrow{s,r}{} \arrow{e,t}{\theta^{}_{S^n}} \node{L_{n}(e)} \arrow{s,r}{=}\\
\node{\dots} \arrow{e} \node{L_{n+1}(e)} \arrow{e,t}{\omega_M} \node{\mathcal{S}^{Diff}(M)} \arrow{e,t}{\eta^{}_M} \node{\mathcal{N}^{Diff}(M)} \arrow{e,t}{\theta^{}_{M}} \node{L_{n}(e).}
\end{diagram}
\end{equation}
Moreover if $c: M \rightarrow S^n$ is the collapse map then the action of $\mc{N}^{Diff}(S^n)$ on $\mc{N}^{Diff}(M)$ may be identified with the action of $\pi_n(G/O)$ on $[M, G/O]$ given via $c^*: \pi_n(G/O) \rightarrow [M, G/O]$.
\end{Lemma}

\begin{proof}
This is just a question of checking definitions.  In particular the last statement can be seen by another application of \cite{M-M}[Theorem 2.23].
\end{proof}

\begin{Corollary} \label{Theta/bPcor}
Let $M$ be a closed smooth $n$-manifold such that the top cell of $M$ splits off after stabilisation.  Then $\Theta_{n}/bP_{n+1}$ acts freely on $\mathcal{S}^{Diff}(M)/bP_{n+1}$.
\end{Corollary}

\begin{proof}
Our assumption on $M$ ensures that $c^* : \pi_n(G/O) \rightarrow [M, G/O]$ is a split injection where $c$ is the collapse map.  Thus by Lemma \ref{sescomlem} the action of $\Theta_n/bP_{n+1}$ is detected via $\eta_M$ in $\pi_{n}(G/O) \supset \Theta_{m+n}/bP_{m+n+1}$: if $[\Sigma] \in \Theta_{n}$ and $[N, f] \in \mathcal{S}^{Diff}(M)$ then 
\[\eta_{M}([\Sigma \sharp N, f]) = \eta_{M}([N, f]) + c^*\eta_{S^{n}}([\Sigma]) ~ \in ~ [M, G/O] .\]
\end{proof} 

We now turn to the action of $bP_{n+1} \subset \Theta_n$ on $\mathcal{S}^{Diff}(M)$.  By Browder, \cite{Bro}[II 4.10, 4.11], the action of $L_{n+1}(e)$ on $\mathcal{S}^{Diff}(M)$ factors through $\omega_{S^n} : L_{n+1}(e) \sra bP_{n+1}$ and the action of $bP_{n+1} \subset \Theta_n$ on $\mathcal{S}^{Diff}(M)$ described above.   The surgery exact sequence
\[ \dots \dlra{} \mathcal{N}^{Diff}(M \times [0, 1] ) \dlra{\theta_{M \times [0, 1]}}L_{n+1}(e) \dlra{\omega^{Diff}} \mathcal{S}^{Diff}(M) \dlra{} \dots \]
allows us to calculate the stabilisers of these actions on the base-point: 
\begin{equation} \label{L&bPacteqn}
L_{n+1}(e)_{[{\rm Id}_{}]} = {\rm Im}(\theta_{M \times [0, 1]}) ~~~\text{and so}~~~(bP_{n+1})_{[{\rm Id}_{}]} = \omega_{S^n}({\rm Im}(\theta_{M \times [0, 1]})).
\end{equation}
To calculate other stabilisers we shall use the following lemma on the naturality of the surgery exact sequence.  Although at present we consider simply-connected manifolds, a little later we will encounter oriented manifolds with fundamental group $\Z$.  We therefore state the Lemma for general fundamental groups.

\begin{Lemma} \label{sesnatlem}
A homotopy equivalence  $g: Y \rightarrow Z$ between closed $Cat$-manifolds with arbitrary fundamental groups induces a commutative diagram
\[
\begin{diagram}
\node{L_{n+1}(\pi)} \arrow{s,r}{g_*} \arrow{e,t}{\omega_Y} \node{\mathcal{S}^{Cat}(Y)} \arrow{s,r}{g_*} \arrow{e,t}{\eta_Y} \node{\mathcal{N}^{Cat}(Y)} \arrow{e,t}{\theta_Y} \arrow{s,r}{g_*} \node{L_q(\pi)} \arrow{s,r}{g_*} \\
\node{L_{n+1}(\pi)} \arrow{e,t}{\omega_Z} \node{\mathcal{S}^{Cat}(Z)} \arrow{e,t}{\eta_Z} \node{\mathcal{N}^{Cat}(Z)} \arrow{e,t}{\theta_Z} \node{L_{n}(\pi).}
\end{diagram}
\]
\end{Lemma}

\begin{proof}
We give only the definitions of the maps and leave the details to the reader: if $[N, f] \in \mc{S}^{Cat}(Y)$ then $g_*([N, f]) = [N, g \circ f]$, if $[N, f, b] \in \mc{N}^{Cat}(Y)$ then $g_*([N, f, b]) = [N, g \circ f, g^{-1*} \circ b]$ and $g_* : L_i(e) \cong L_i(e)$ is the map induced by $g_* : \pi_1(Y) \cong \pi_1(Z)$.
\end{proof}

\begin{Lemma} \label{bPactlem}
Let $[N, f] \in \mathcal{S}^{Diff}(M)$.  The stabilisers of $[N,f]$ under the actions of $L_{n+1}(e)$ and $bP_{n+1}$ are given by
\[ L_{n+1}(e)_{[N, f]} = {\rm Im}(\theta^{}_{N \times [0, 1])}) ~~~\text{and }~~~(bP_{n+1})_{[N, f]} = \omega_{S^n}({\rm Im}(\theta_{N \times [0, 1]})).
\]
\end{Lemma}

\begin{proof}
Let $f^{-1} : M \rightarrow N$ be a homotopy inverse for $f$ and consider the commutative diagram of Lemma \ref{sesnatlem} applied to $f^{-1}$.  By definition, $f^{-1}_*([N, f]) = [{\rm Id_N}]$ and so $f^{-1}_* : L_{n+1}(e) = L_{n+1}(e)$ carries the stabiliser of $[N, f]$ isomorphically to the stabiliser of $[{\rm Id}_N]$ which we determined above in $(\ref{L&bPacteqn})$.  The statement for the action of $bP_{n+1}$ follows because $L_{n+1}(e)$ acts via the surjection $\omega_{S^n} : L_{n+1}(e) \rightarrow bP_{n+1}$.
\end{proof}


Given a homotopy equivalence $f: N \rightarrow M$, we would like to use functorality to relate $\theta_{M \times [0,1]}$ and $\theta_{N \times [0, 1]}$.  But as $f$ is merely a homotopy equivalence, composition with $f$ does not induce a map $\mathcal{N}^{Cat}(N \times [0, 1]) \rightarrow \mathcal{N}^{Cat}(M \times [0, 1])$: we need a $Cat$ isomorphism on the boundary.  We therefore find a way to return to the closed case.  Given a rel. boundary degree on normal map $(g, c) : W \rightarrow M \times [0, 1]$ we may always assume that $(g, c)$ is the identity over the boundary of $W$.  Gluing the ends of both domain and range together we obtain a well defined map 
\begin{equation} \label{Teqn}
T_M : \mathcal{N}^{Cat}(M \times [0, 1] ) \rightarrow \mathcal{N}^{Cat}(M \times S^1 ). 
\end{equation}
For the next lemma, recall that $\mc{N}^{Cat}(M \times [0, 1]) \equiv [\Sigma M, G/Cat]$ where $\Sigma M$ is the suspension of $M$.

\begin{Lemma}  \label{Tthetacomlem}   
With the identifications $\mathcal{N}^{Cat}(M \times [0, 1]) \equiv [\Sigma M, G/Cat]$ and $\mathcal{N}^{Cat}(M \times S^1) \equiv [M, G/Cat] \times [\Sigma M, G/Cat]$, the map $T_M$ is the inclusion onto the second factor.  Moreover for the canonical inclusion $L_{n+1}(e) \rightarrow L_{n+1}(\Z)$ the following diagram commutes
\[
\begin{diagram}
\node{\mathcal{N}^{Cat}(M \times [0, 1] )} \arrow{e,t}{\theta_{M \times [0,1]} }\arrow{s,r}{T_M} \node{L_{n+1}(e)} \arrow{s,r}{=} \\
\node{\mathcal{N}^{Cat}(M \times S^1  )} \arrow{e,t}{\theta_{M \times S^1} }\node{L_{n+1}(\Z).}  
\end{diagram}
\]
\end{Lemma}

\begin{proof}
Let $i : M = M \times \{ pt \} \rightarrow M \times S^1$ be the inclusion and identify $(M \times S^1)/i(M) \equiv \Sigma M \vee S^1$.  Then since $\pi_1(G/Cat) = 0$ we have a short exact sequence
\[  [\Sigma M, G/Cat] \dlra{} [M \times S^1, G/Cat] \dlra{i^*} [M, G/Cat] \]
where $i^*([V, f, b]) = 0$ if and only if $(V, f, b)$ is normally bordant to a degree one normal map which splits along $M \times \{ pt \} \subset M \times S^1$ as a $Cat$-isomorphism.  This is certainly true for any $T_M([W, g, c])$.  On the other hand, if $i^*([V, f, b]) = 0$ we see that $[V, f, b]$ lies in the image of $T_M$.

The second statement uses Shaneson splitting, \cite{Sh} : $L_{n+1}(\Z) \cong L_{n+1}(e) \oplus L_{n}(e)$.  As $T_M([W, g, c])$ splits along $M \times \{ pt \}$, $\theta_{M \times S^1}(T_M([W, g, c]))$ lies in the $L_{n+1}(e)$ component of $L_{n+1}(\Z)$.  Moreover, the surgery kernels are the same for $(W, g, c)$ and $T_M(W, g, c)$: homologically, if  $K_*$ is the surgery kernel of $(W, g, c)$ then $H_*(W) \cong K_* \oplus H_*(M \times [0, 1])$ and $H_*(T_M(W)) \cong K_* \oplus H_*(M \times S^1)$.  It follows that the surgery obstructions for $(W, g, c)$ and $(T_M(W, g, c))$ are the same element of $L_{n+1}(e)$.

\end{proof}

For a homotopy equivalence $f: N \rightarrow M$ the homotopy equivalence $f \times {\rm Id} : N \times S^1 \rightarrow M \times S^1$ induces the following commutative diagram by Lemma \ref{sesnatlem}:
\begin{equation} \label{S^1comeqn}
\begin{diagram}
\node{\mathcal{N}^{Cat}(N \times S^1)} \arrow{e,t}{\theta^{Cat}_{N \times S^1}} \arrow{s,r}{(f \times {\rm Id})_*} \node{L_{n+1}(\Z)} \arrow{s,r}{=} \\
\node{\mathcal{N}^{Cat}(M \times S^1)} \arrow{e,t}{\theta^{Cat}_{M \times S^1}} \node{L_{n+1}(\Z).}  
\end{diagram}
\end{equation}
\noindent
We now combine the map $T_M$ and the above diagram to obtain our key lemma.

\begin{Lemma} \label{keysurgoblem}
Let $f: N \rightarrow M$ be a homotopy equivalence and let $x \in \mathcal{N}^{Cat}(N \times [0, 1])$.  The surgery obstruction maps $\theta_{M \times S^1}$ and $\theta_{N \times [0,1]}$ are related by 
\[ \theta_{N \times [0, 1]}(x) = \theta_{M \times S^1}(\eta([N \times S^1, f \times {\rm Id}_{}]) + (f \times {\rm Id}_{})^{-1*}(T_N(x))) ~ \in~ L_{n+1}(e).\]
\end{Lemma}

\begin{proof}
By Lemma \ref{sesnatlem} and $(\ref{S^1comeqn})$ we have the equalities
\[\theta_{N \times [0, 1]} = \theta_{N \times S^1} \circ T_N = \theta_{M \times S^1} \circ (f \times {\rm Id})_* \circ T_N.\]
We now require the following composition formula for normal invariants of \cite{Bru, M-T-W, Ra4}.


\begin{Proposition} \label{compoprop}
Let $g: Y \rightarrow Z$ be a homotopy equivalence of $Cat$ manifolds and let $(f, b): X \rightarrow Y$ be a degree one normal map representing $[X, f, b] \in [Y, G/Cat]$.  Then
\[ g_*([X, f, b]) = \eta^{Cat}([Y, g]) + g^{-1*}([X, f, b]) \in [Z, G/Cat].\]
\end{Proposition}

\begin{Remark}
The above formula goes back to \cite{Bru} where it was proven in the category of punctured, simply-connected, smooth manifolds and to \cite{M-T-W} where it was proven for topological manifolds with non-empty boundary.  From \cite{Ra4} we have it in the topological category for manifolds with or without boundary.  However one may check that the arguments of \cite{Bru, M-T-W} work equally well for closed manifolds of any fundamental group and of any category.
 \end{Remark}

\noindent
Now given $x \in \mathcal{N}^{Cat}(N \times [0, 1])$ we apply Proposition \ref{compoprop} to $f \times {\rm Id} : N \times S^1 \rightarrow M \times S^1$:
\[ \theta^{}_{N \times [0,1]}(x) = \theta^{}_{M \times S^1}((f \times {\rm Id})_*(T_N(x))) = \theta^{}_{M \times S^1}(\eta^{}([N \times S^1, f \times {\rm Id}] + (f \times {\rm Id})^{-1*}(T_N(x))).\]
%
\end{proof}

\begin{Remark} \label{rewriterem}
A homotopy equivalence $f: N \rightarrow M$ does define the following map
\[f_\times :  \mathcal{N}^{Cat}(N \times [0, 1]) \equiv [\Sigma N, G/Cat] \dlra{f^{-1*}} [\Sigma M, G/Cat] \equiv \mathcal{N}^{Cat}(M \times [0,1]). \]
It is not hard to check that $(f \times {\rm Id})^{-1*}(T_N(x)) = T_M(f_\times(x))$ and so we could rewrite the equation of Lemma \ref{keysurgoblem} as 
\[ \theta_{N \times [0, 1]}(x) = \theta_{M \times S^1}(\eta([N \times S^1, f \times {\rm Id}_{}]) \times T_M(f_\times(x))) ~\in~ L_{n+1}(e). \]
\end{Remark}

\begin{Corollary} \label{stabhomeodepcor}
For $[N, f] \in \mc{S}^{Diff}(M)$, the stabilisers $L_{n+1}(e)_{[N, f]}$ and $(bP_{n+1})_{[N, f]}$ depend only upon the homeomorphism type of $N$.
\end{Corollary}

\begin{proof}
This follows quickly from Remark \ref{rewriterem} and the fact that $\theta^{Diff}_{N \times [0, 1]} = \theta^{Top}_{N \times [0, 1]} \circ F$ and also the fact that $\eta^{Top}([N \times S^1, f \times {\rm Id}]) = 0$ if $f$ is a homeomorphism.
\end{proof}


\section{The proofs of Theorems \ref{mainthm1} and \ref{mainthm2}}  \label{proofsec}
\begin{proof}[Proof of Theorem \ref{mainthm1}]
We shall move along the sequence of Theorem \ref{mainthm1} from left to right: by \cite{DS} or \cite{Sc} the inertia group of $S^p \times S^q$ is trivial and thus $\Theta_{p+q}$ acts freely on the base point of $\mathcal{S}^{Diff}(S^p \times S^q)$.  

We next determine the formula for the surgery obstruction map $\theta_{S^p \times S^q}^{}$.  Recall that $i: S^p \vee S^q \rightarrow S^p \times S^q$ is the inclusion and that $c : S^p \times S^q \rightarrow S^{p+q}$ is the collapse map.  There is a short exact sequence
\begin{equation}
0 \dlra{} \pi_{p+q}(G/Cat) \dlra{c^*} [S^p \times S^q, G/Cat] \dlra{i^*} \pi_p(G/Cat) \times \pi_q(G/Cat) \dlra{} 0
\end{equation}
and the product map $\pi_{p, q} : \mc{N}^{Cat}(S^p) \times \mc{N}^{Cat}(S^q)$ from Proposition \ref{surgprodprop} gives the map
\[ \pi_{p, q} : \pi_p(G/Cat) \times \pi_q(G/Cat) \equiv \mathcal{N}^{Cat}(S^p) \times \mathcal{N}^{Cat}(S^q) \dlra{} \mathcal{N}^{Cat}(S^p \times S^q).\]
We leave the reader to check that $i^* \circ \pi_{p, q}$ is the identity on $\pi_p(G/Cat) \times \pi_q(G/Cat)$ and henceforth use $\pi_{p, q} \times c^*$ to make the identification
\[ (\pi_p(G/Cat) \times \pi_q(G/Cat) ) \times \pi_{p+q}(G/Cat) \dlra{ \pi_{p, q} \times c^*} [S^p \times S^q, G/Cat] \equiv \mathcal{N}^{Cat}(S^p \times S^q).\] 
Before stating the next lemma we recall the isomorphism $\pi_i(G/Top) \cong L_i(e)$, the product  $L_{p}(e) \tensor L_{q}(e) \rightarrow L_{p+q}(e)$ and the homomorphism $F: \pi_i(G/O) \rightarrow \pi_i(G/Top)$ which may be identified with the surgery obstruction map $\theta^{Diff}_{S^i} : \pi_i(G/O) \rightarrow L_i(e)$.  It easy to check that $F: \mathcal{N}^{Diff}(S^p \times S^q) \rightarrow \mathcal{N}^{top}(S^p \times S^q)$ can be identified with the product of the induced maps on homotopy groups $F : \prod_{i=1}^3\pi_{p_i}(G/O) \rightarrow \prod_{i=1}^3\pi_{p_i}(G/Top)$, $(p_1, p_2, p_3) = (p, q, p+q)$.

\begin{Lemma} \label{surgoblem1}
With the above identification $\mc{N}^{Cat}(S^p \times S^q) \equiv \prod_{i=1}^{3}\pi_{p_i}(G/Cat)$, the surgery obstruction maps $\theta_{S^p \times S^q}^{Top}$ and $\theta_{S^p \times S^q}^{Diff}$ are given by:
\[ \theta^{Top}_{S^p \times S^q} :   \pi_p(G/Top) \times \pi_q(G/Top) \times \pi_{p+q}(G/Top)  \dlra{}  L_{p+q}(e), ~~~(x, y, z) \mapsto   xy + z, \]
\[\theta^{Diff}_{S^p \times S^q} :   \pi_p(G/O) \times \pi_q(G/O) \times \pi_{p+q}(G/O)  \dlra{}  L_{p+q}(e), ~~~(u, v, w) \mapsto  F(u)F(v) + F(w).\]
\end{Lemma}

\begin{proof}
We give two proofs, the first starting with the smooth category: for $(u, v, 0) = \pi_{p, q}(u, v)$, we apply the surgery product formula of Proposition \ref{surgprodprop} and the fact that the symmetric signatures $\sigma^*(S^p)$ and $\sigma^*(S^q)$ vanish.  The normal bordism class $(u, v, w)$ is equal to that of $(u, v, 0)$ acted upon by $(0, 0, w) = c^*(w)$ for $w \in \pi_{p+q}(G/O)$ and we apply the commutativity of the diagram of Lemma \ref{sescomlem} to complete the proof.  The same proof works in the topological case where a version of Lemma \ref{sescomlem} of course holds.

A second proof may be given using \cite{Ra4}[Example 3.6(i)] where $\theta_{S^p \times S^q}^{Top}$ is given in terms of the alternative group structure on $\mc{N}^{Top}(S^p \times S^q)$.  Combining Ranicki's formula for the surgery obstruction map with his formula comparing the two group structures on $\mc{N}^{Top}(S^p \times S^q)$ gives the same answer for $\theta^{Top}_{S^p \times S^q}$ and the smooth case follows from the fact that $\theta^{Diff} = \theta^{Top} \circ F$.
\end{proof}
\noindent


%
%
%
%

From Lemma \ref{surgoblem1} we see that the image of $\eta^{Diff} : \mathcal{S}^{Diff}(S^p \times S^q) \rightarrow \mathcal{N}^{Diff}(S^p \times S^q)$ is the set $\{ (u, v, w) | F(u)F(v) -F(w) = 0 \}$.  In particular, if $(u, v) \in {\rm Im}(i^* \circ \eta^{Diff})$ then the set of $w$ such that $(u, v, w)$ lies in the image of $\eta^{Diff}$ is a coset of ${\rm Ker}(F: \pi_{p+q}(G/O) \rightarrow \pi_{p+q}(G/Top))$.  Applying the exactness of the surgery exact sequence for $S^{p+q}$ and the commutative diagram of Lemma \ref{sescomlem} we see that $\Theta_{p+q}$ acts transitively on $(i^* \circ \eta^{Diff})^{-1}(u, v) \subset \mathcal{S}^{Diff}(S^p \times S^q)$.

We now characterise the image of $i^* \circ \eta^{Diff} : \mathcal{S}^{Diff}(S^p \times S^q) \rightarrow  \pi_p(G/O) \times \pi_q(G/O)$.   Recall that $t_i$ is the non-negative integer defined by $t_i = 0$ if $i$ is not divisible by $4$ and $F(\pi_{4k}(G/O)) = t_{4k} \cdot \pi_{4k}(G/Top)$.
Recall also that we have chosen a splitting $\tau_{4k} \times \phi_{4k} : \pi_{4k}(G/O) \cong \Theta_{4k} \times \Z$.  For $i \neq 4k$ we define $\phi_{i} : \pi_i(G/O) \rightarrow 0$ to be the zero map and for all $i$ we choose fixed generators $z_i$ for $L_i(e)$ such that $F(\phi^{-1}_{4k})(1) = t_{4k} \, z_{4k}$ when $i=4k$.
With these conventions Lemma \ref{surgoblem1} yields
\[ \theta^{Diff}_{S^p \times S^q} \circ \pi_{p, q} :  \pi_p(G/O) \times \pi_q(G/O) \rightarrow L_{p+q}(e), ~~~(u, v) \mapsto \pm  8 \, \phi_p(u) \, \phi_q(v) \, t_p \, t_q \, z_{p+q} \in L_{p+q}(e).
\]
Applying Lemma \ref{surgoblem1} again we see that the map $\del$ of Theorem \ref{mainthm1} is  the bilinear map 
\[ \del: \pi_p(G/O) \times \pi_q(G/O) \rightarrow 8 \, t_p \,  t_q \cdot bP_{p+q}(e), ~~~(u, v) \mapsto (\omega_{S^{p+q}} \circ \theta^{Diff} \circ \pi_{p, q})(u, v).
\]
\noindent
We saw above that $(u, v) \in {\rm Im}(i^* \circ \eta^{Diff})$ if and only if there is a $w \in \pi_{p+q}(G/O)$ such that $F(u) F(v) = -F(w)$.  It follows that $(u, v) \in {\rm Im}(i^* \circ \eta^{Diff})$ if and only if $\del(u, v) = 0$.  Finally, as $\pi_{4i}(G/O)$ maps onto $t_{4i} \cdot \pi_{4i}(G/Top)$, $\del$ maps $\pi_p(G/O) \times \pi_q(G/O)$ onto $8 \, t_p \, t_q \cdot bP_{p+q}(e)$.  
\end{proof}

\begin{proof}[Proof of Theorem \ref{mainthm2}]  
The top cell of $S^p \times S^q$ splits off after just one suspension and so by Corollary \ref{Theta/bPcor} $\Theta_{p+q}/bP_{p+q+1}$ acts freely on $\mc{S}^{Diff}(S^p \times S^q)/bP_{p+q+1}$.  So the exact sequence of Theorem \ref{mainthm2} follows from the exact sequence of Theorem \ref{mainthm1} and the expression for $\theta^{Diff}_{S^p \times S^q}$ in Lemma \ref{surgoblem1}.  In particular If $p+q \neq 4k$ then $\bar \eta : \mc{S}^{Diff}(S^p \times S^q) \rightarrow \Theta_{p+q}/bP_{p+q+1} \times \pi_p(G/O) \times \pi_q(G/O)$ is just the map $\eta^{Diff}$ onto its image where we have identified $\mc{N}^{Diff}(S^p \times S^q) = \prod_{i=1}^{3}\pi_{p_i}(G/O)$ as in the proof of Theorem \ref{mainthm1} and where we also identify $\Theta_{p+q}/bP_{p+q+1}$ with ${\rm Ker}(\theta^{Diff}_{S^{p+q}} : \pi_{p+q}(G/O) \rightarrow L_{p+q}(e))$.  When $p+q = 4k$ we use the splitting $\tau_{4k} \times \phi_{4k}: \pi_{4k}(G/O) \cong \Theta_{4k} \times \Z$ and must only note that $\theta^{Diff}_{S^{4k}} : \pi_{4k}(G/O)\rightarrow L_{4k}(e)$ factors through $\phi_{4k}$.

It remains to determine the action of $bP_{p+q}$ on $\mathcal{S}^{Diff}(S^p \times S^q)$.  If $p+q$ is odd then $bP_{p+q} = 0$ and there is nothing to prove.  So we assume that $p$ is odd and $q$ is even and begin by determining $\theta^{}_{S^p \times S^q \times S^1}$.  Once again we use products to identify the normal invariant set.  We have the map
\[\pi_{p \times 1, q} : \mathcal{N}^{Cat}(S^p \times S^1) \times \mathcal{N}^{Cat}(S^q) \rightarrow \mathcal{N}^{Cat}(S^p \times S^q \times S^1) \]
and for the inclusion $j : (S^p \times S^q \times S^1)-D^{p+q+1} \rightarrow S^p \times S^q \times S^1$ we have the short exact sequence
 \[ \pi_{p+q+1}(G/Cat) \dlra{c^*} [S^p \times S^q \times S^1, G/Cat] \dlra{j^*} [((S^p \times S^q \times S^1)-D^{p+q+1}), G/Cat]\] 
and a canonical identification
\[ [((S^p \times S^q \times S^1)-D^{p+q+1}), G/Cat] \cong [S^p \times S^1, G/Cat] \times [S^q \times S^1, G/Cat]. \]
In the topological case $\pi_p(G/Top) = 0 = \pi_{q+1}(G/Top)$ so that $[S^p \times S^1, G/Top] \cong \pi_{p+1}(G/Top)$, $[S^q \times S^1, G/Top] \cong \pi_q(G/Top)$ and $\pi_{p \times 1, q}$ defines a section for $j^*$.

\begin{Lemma} \label{nrmlinvtlem}
The map $(\pi_{p \times 1, q} \times c^*)^{-1}$ gives an identification
\[ \mathcal{N}^{Top}(S^p \times S^q \times S^1) \equiv L_{p+1}(e) \times L_{q}(e) \times L_{p+q+1}(e)  \]
such that the surgery obstruction map has the following form
\[ \theta^{Top}_{S^p \times S^q \times S^1} : L_{p+1}(e) \times L_{q}(e) \times L_{p+q+1}(e)  \dlra{} L_{p+q+1}(e), ~~~(x,y,z) \mapsto x \cdot y + z \]
and such that the map $T_{S^p \times S^q}$ from $(\ref{Teqn})$ takes the form
\[ T_{S^p \times S^q}: L_{p+1}(e) \times L_{p+q+1}(e) \dlra{} L_{p+1}(e) \times L_{q}(e) \times L_{p+q+1}(e),~~~(x, z) \mapsto (x, 0, z)  .\]
\end{Lemma}

\begin{proof}
The identification of $\theta_{S^p \times S^q \times S^1}^{Top}$ follows from Proposition \ref{surgprodprop} and that fact that the symmetric signatures $\sigma^*(S^p \times S^1)$ and $\sigma^*(S^q)$ vanish.  The second statement may require a little comment: denote by $P_q$ both the maps $ \mathcal{N}^{Top}(S^p \times [0, 1]) \rightarrow \mathcal{N}^{Top}(S^p \times S^q \times [0, 1])$ and $\mathcal{N}^{Top}(S^p \times S^1) \rightarrow \mathcal{N}^{Top}(S^p \times S^q \times S^1)$ which multiply the domain and range of a degree one normal map with ${\rm Id} : S^q = S^q$.  Under the identification 
\[ \mathcal{N}^{Top}(S^p \times S^q \times [0, 1]) = [S^{p+1} \vee S^{q+1} \vee S^{p+q+1}, G/Top] \]
the summand $\pi_{p+1}(G/Top)$ is equal to $P_q(\mathcal{N}^{Top}(S^p \times [0, 1]))$. The result follows from the easily verified identity $T_{S^p \times S^q} \circ P_q = P_q \circ T_{S^p}$.
\end{proof}
\noindent Now let $i_q : S^q \rightarrow S^p \times S^q$ be the inclusion and recall that $z_{i}$ is a fixed generator of $\pi_{i}(G/Top)$.  For $[N, f] \in \mc{S}^{Diff}(S^p \times S^q)$ we extend the definition of  $d([N, f]) \in \Z$ from Theorem \ref{mainthm2} by
\[ d([N, f]) = \left\{ \begin{array}{cl} 0 & \text{~if $q = 2$ mod $4$} \\ F( (i^*_{4k} \circ \eta^{Diff})([N,f]))) = d([N, f]) \, t_{4k} \, z_{4k} \in \pi_{4k}(G/Top) & \text{~if $q = 4k$.}
 \end{array} \right. \]


\begin{Lemma} \label{thetatoplem}
Let $[N, f] \in \mc{S}^{Diff}(S^{p} \times S^{q})$.  Then we have the following formulae for the surgery obstruction maps $\theta^{Top}_{N \times [0, 1]}$ and $\theta^{Diff}_{N \times [0, 1]}$:
\[ \theta^{Top}_{N \times [0,1]} : L_{p+1}(e) \times L_{p+q+1}(e) \rightarrow L_{p+q+1}(e), ~~(x, z) \mapsto d([N, f]) \, t_{q} \, x \, z_q + z,\]
%
\[ \theta^{Diff}_{N \times [0,1]} : \prod_{i=1}^3 \pi_{p_i+1}(G/O) \rightarrow L_{p+q+1}(e), ~~(u, v, w) \mapsto d([N, f]) \, t_{q} \, F(u) \, z_q  + F(w).\]
\end{Lemma}

\begin{proof}
We first identify $\eta^{Top}([N \times S^1, f \times {\rm Id}]) \subset [S^1 \times S^p \times S^q, G/Top]$.  As $p$ is odd and $q$ is even, $\eta^{Top}([N, f]) \in \pi_q(G/Top) \subset \mathcal{N}^{Top}(S^p \times S^q)$ and so $\eta^{Top}([N \times S^1, f \times {\rm Id}]) = \pi_{p\times 1, q}(0, \eta^{Top}([N, f]))$.  The topological case now follows from Lemmas \ref{grpprodlem}, \ref{keysurgoblem} and \ref{nrmlinvtlem}.  The smooth case follows from the topological case and the fact that $\theta^{Diff}_{N \times [0,1]}  = \theta^{Top}_{N \times [0,1]}  \circ F$.
\end{proof}

We now complete the proof of Theorem \ref{mainthm2}. Let $[N, f] \in \mc{S}^{Diff}(S^p \times S^q)$.  If $(p, q) \neq (4j-1, 4k)$ then by Lemma \ref{thetatoplem}, ${\rm Im}(\theta^{Diff}_{N \times [0, 1]}) = F(\pi_{p+q+1}(G/O)) \subset L_{p+q+1}(e)$ and so by Lemma \ref{bPactlem}, $bP_{p+q+1}$ acts freely on all of $\mc{S}^{Diff}(S^p \times S^q)$.  However, when $(p, q) = (4j-1, 4k)$ Lemma \ref{thetatoplem} shows that ${\rm Im}(\theta^{Diff}_{N \times [0, 1]}) \subset L_{4(j+k)}(e)$ is the subgroup generated by $t_{4(j+k)} \, z_{4(j+k)}$ and $8 \, d([N, f]) \, t_{4k} \, t_{4j} \, z_{4(j+k)}$.  Applying Lemma \ref{bPactlem} again we conclude that the stabiliser of the action of $bP_{4(j+k)}$ on $[N, f]$ is $8 \, d([N,f]) \, t_{4j} \, t_{4k} \cdot bP_{4(j+k)}$.
\end{proof}

\section{The order of $bP_{4k}$} \label{numsec}
In this short section we review the computation the order of $bP_{4k}$.
\begin{Theorem} [\cite{Le}]
For $k \geq 2$, the group $bP_{4k}$ 
is a cyclic group of order 
\[ t_{4k} = a_k(2^{2k-2})(2^{2k-1}-1){\rm Num}(B_k/4k) \]
where $a_k = (3-(-1)^k)/2$, $B_k$ is the $k$-th Bernoulli number (topologist's indexing) and ${\rm Num}(B_k/4k)$ is the numerator of the fraction $B_k/4k$ expressed in reduced form.
\end{Theorem}
Recalling that $t_{4} := |{\rm Cok}(\pi_{4}(G/O) \rightarrow \pi_{4}(G/Top))| = 2$ we have the following table: 
\[ t_4 = 2, ~~~ t_8 = 28 = 4 \times 7,~~~ t_{12} = 992 = 32 \times 31, ~~~t_{16} = 8182 = 64 \times 127. \]
\noindent
From this we obtain the following calculation of $8 \, t_{4k}\, t_{4j}\cdot bP_{4(j+k)}$ for small values of $j$ and $k$:
\[ 8 \, t_4 \, t_4\cdot bP_8 \cong \Z_7,~~~ 8 \, t_4 \, t_8 \cdot bP_{12} \cong \Z_{31}, ~~~ 8 \, t_4 \, t_{12} \cdot bP_{16}, \cong \Z_{127}, ~~~ 8 \, t_8 \, t_8\cdot bP_{16} = \Z_{127}.\]
Finally, it is well known that ${\rm Num}(B_k/4k)$ is odd and as a consequence we see that the $2$-primary component of $8 \, t_{4j}\, t_{4k} \cdot bP_{4(j+k)}$ is always trivial.




\noindent
\textsc{Diarmuid Crowley\\
School of Mathematical Sciences \\
University of Adelaide \\
Australia, 5005.}\\ \\
{\it E-mail address:} \texttt{diarmuidc23@gmail.com}

\end{document}